\documentclass[a4paper,12pt,3p]{elsarticle}

\usepackage{amssymb,amsthm,graphicx,physics,mathtools,esint,cancel,mathrsfs,xcolor,enumerate}
\usepackage[colorlinks]{hyperref}

\hypersetup{colorlinks, citecolor=blue, filecolor=blue, linkcolor=blue, urlcolor=blue}

\newcommand{\N}{\mathbb{N}}

\newcommand{\R}{\mathbb{R}}
\newcommand{\C}{\mathbb{C}}
\renewcommand{\phi}{\varphi}
\let\le\leqslant

\DeclareMathOperator{\supp}{supp}
\let\span\relax
\DeclareMathOperator{\span}{span}

\DeclareMathOperator{\Ran}{Ran}

\DeclarePairedDelimiter{\parens}{\lparen}{\rparen}
\let\braces\relax
\DeclarePairedDelimiter{\braces}{\{}{\}}

\newtheorem*{Th*}{Theorem}
\newtheorem{Th}{Theorem}
\newtheorem*{St*}{Statement}

\newtheorem*{Lm*}{Lemma}
\newtheorem{Lm}{Lemma}
\newtheorem*{Cr*}{Corollary}

\newdefinition{Df}{Definition}
\newtheorem*{Xr*}{Exercise}

\newdefinition{Rm}{Remark}
\newtheorem*{Cn*}{Conjecture}

\journal{Journal of Functional Analysis}

\begin{document}

\begin{frontmatter}

\title{On a conjecture of de Branges}

\author{Igor Bereza\corref{cor1}}
%\ead{ibereza@disroot.org}
\affiliation{
    organization={Department of Mathematics and Computer Science, St. Petersburg State University},
    addressline={7/9 Universitetskaya nab.},
    city={St. Petersburg},
    postcode={199034},
    country={Russia}
}

\begin{abstract}
We provide a broad class of counterexamples to a conjecture of L. de Branges concerning the superfluity of the continuity property in the axiomatic description of de Branges spaces.
\end{abstract}


\end{frontmatter}

\section{Introduction}
We study the minimality of the following axiomatic description of a class of Hilbert spaces of entire functions called de Branges spaces (known as de Branges axioms) introduced in \cite{deBranges59}:
\begin{enumerate}
    \item[(H1)] Whenever $F$ is in the space and has a non-real zero $w$, the function $F(\cdot)(\cdot - \bar{w})(\cdot - w)^{-1}$ belongs to the space and has the same norm.
    \item[(H2)] For every complex number $w$, the linear functional defined on the space by $F \mapsto F(w)$ is continuous.
    \item[(H3)] Whenever $F$ is in the space, the function $F^*(z) = \overline{F(\bar{z})}$ belongs to the space and has the same norm.
\end{enumerate}

The aim of this article is to show that it is possible to transform any de Branges space from a suitable class into a Hilbert space of entire functions that violates (H2), but satisfies (H1) and (H3).

One of the most useful properties of de Branges spaces is \cite[Theorem 23]{deBranges68}, given below. It states that de Branges spaces admit the following constructive definition, which is often easier to work with once it is established that the axioms are fulfilled.

\begin{Df}
    An entire function $E$ belongs to the Hermite-Biehler class if $|E^*(z)| < |E(z)|$ for all $z \in \C^+$.
\end{Df}

\begin{Df}
    The de Branges space $H(E)$ associated with a Hermite-Biehler entire function $E$ is the set of all entire functions $F$ such that 
    \[ ||F||^2_E = \int\limits_\R \bigg|\frac{F(t)}{E(t)}\bigg|^2 \dd t < +\infty \]
    and $F/E$, $F^*/E$ belong to $H^2(\C^+)$, the Hardy space on the upper half-plane.
\end{Df}

\begin{Th}[de Branges]\label{EquivalenceOfDefinitions}
    A Hilbert space of entire functions $\mathcal{H} \ne \{0\}$ satisfies (H1) -- (H3) iff there exists an entire function $E$ of the Hermite-Biehler class such that $\mathcal{H} = H(E)$ and $||\cdot||_\mathcal{H} = ||\cdot||_E$.
\end{Th}

We shall not use the constructive definition, it is mentioned only as a part of a brief historical account.
Now we turn our attention to the axiomatic description of de Branges spaces.
\newline
\indent Various aspects of independence of the de Branges axioms have been studied before. For example, L. de Branges showed that the following result holds.
Although his work \cite[Problem 54]{deBranges68} includes only the statement, one may find a detailed proof, for instance, in \cite{Romanov24}.
\begin{St*}[de Branges]
    Let $\mathcal{H} \ne \{0\}$ be a Hilbert space of entire functions which satisfies (H1) and (H2).
    Then there exists a de Branges space $H(E)$ and an entire function $U$ with $UU^* \equiv 1$ such that $F \mapsto UF$ is an isometric transformation of $\mathcal{H}$ onto $H(E)$.
\end{St*}

We are concerned with independence of the axiom (H2). The subject in question was considered in the original de Branges' papers: in \cite{deBranges59} L. de Branges showed that Theorem \ref{EquivalenceOfDefinitions} still holds if one weakens the axiom (H2) by requiring the continuity of the point evaluation functionals only for non-real points. Moreover, at the bottom of page 445 of \cite{deBranges63} L. de Branges posed the following conjecture, which we quote verbatim.
\begin{Cn*}[de Branges]
    The axiom (H2) which appears here is conjectured to be a consequence of (H1).
\end{Cn*}

We shall show that the Conjecture is false using the following class of Hilbert spaces of entire functions.
\begin{Df}
    Let $\mathcal{M}$ be the set of all Borel measures $\mu$ supported on infinite discrete (i.e. having no finite accumulation points) subsets of the real line. We denote by $\{t_i\}_{i \in \N}$ an enumeration of $\supp\mu$ and by $\mu_n$ the mass of $\mu$ at $t_n$.
\end{Df}

\begin{Df}
    A Hilbert space of entire functions $\mathcal{H}$ is isometrically isomorphic to $\ell^2(\mu)$, where $\mu \in \mathcal{M}$, if $||F||^2_\mathcal{H} = \sum_n \mu_n |F(t_n)|^2$ and the map $F \mapsto \{F(t_i)\}_{i \in \N}$ is bijective.
\end{Df}

\begin{Rm}\label{Remark}
    For Hilbert spaces of entire functions that are isometrically isomorphic to $\ell^2(\mu)$ for some $\mu \in \mathcal{M}$ the requirement of equality of norms in the properties (H1) and (H3) is not needed, since $\mu$ is supported on the real line.
\end{Rm}

Our main result, proven in paragraph \ref{ConstructionOfHTilde}, is as follows.

\begin{Th}\label{MainTheorem}
    For every $\mu \in \mathcal{M}$ and a de Branges space $\mathcal{H}$, isometrically isomorphic to $\ell^2(\mu)$, there exists a Hilbert space of entire functions $\widetilde{\mathcal{H}}$ such that
    \begin{itemize}
        \item $\widetilde{\mathcal{H}}$ is isometrically isomorphic to $\ell^2(\mu)$;
        \item $\widetilde{\mathcal{H}}$ satisfies (H1) and (H3), but not (H2);
        \item $\mathcal{H} \cap \widetilde{\mathcal{H}}$ is dense in $\mathcal{H}$ and $\widetilde{\mathcal{H}}$.
    \end{itemize}
\end{Th}

To refute the Conjecture for the aforementioned class of Hilbert spaces, Theorem \ref{MainTheorem} requires us to start with a “bootstrap” de Branges space $\mathcal{H}$, which is isometrically isomorphic to $\ell^2(\mu).$ The following folklore-type result, the proof of which will be given for completeness in paragraph \ref{ConstructionOfXv}, allows us to complete the construction of our counterexamples.

\begin{Th}\label{SecondTheorem}
    For every measure $\mu \in \mathcal{M}$ there exists a de Branges space $\mathcal{H}$ that is isometrically isomorphic to $\ell^2(\mu)$.
\end{Th}

\section{Proofs}

\subsection{Preliminaries}

Our construction is an extension of a standard way of producing a discontinuous functional on a Hilbert space: pick a linearly independent sequence $\{x_n\}$ converging to $x_0$ and set the functional to one on $\{x_n\}$ and to zero on $x_0$, then extend this partially-defined functional in any way to the entire Hilbert space. In this paragraph we develop the necessary machinery for extending linear functionals from a linear subset of Hilbert spaces.

\begin{Df}
    A generating set $Q$ of a vector space $U$ is a subset of $U$ such that every element of $U$ is a finite linear combination of elements of $Q$.
\end{Df}

\begin{Df}
    A Hamel basis of a vector space is a linearly independent generating set.
\end{Df}

\begin{Th}[Existence of a Hamel basis]\label{HamelBasis}
    Let $P$ be a linearly independent set contained in a generating set $Q$, then there exists a Hamel basis $\mathcal{B}$ satisfying $P \subseteq \mathcal{B} \subseteq Q$.
\end{Th}

While in the general case the above Theorem would suffice for extending a linear functional, the space $\widetilde{\mathcal{H}}$ must satisfy (H3). For this purpose we introduce an auxiliary definition of a symmetric linear subset of a Hilbert space and a Lemma which allows us to partition a Hilbert space into two symmetric linear pieces.

\begin{Df}
    We call a linear subset $L \subseteq \ell^2(\mu)$, where $\mu \in \mathcal{M}$, symmetric if $\bar{v} \in L$ when $v \in L$, where $\bar{v}$ denotes $\{\overline{v_i}\}_{i \in \N}$.
\end{Df}

\begin{Lm}\label{SymmetricLinearComplement}
    For any symmetric linear subset $L \subseteq \ell^2(\mu)$, where $\mu \in \mathcal{M}$, there exists a symmetric linear subset $M \subseteq \ell^2(\mu)$, such that every element of $\ell^2(\mu)$ can be uniquely expressed as a sum of elements of $L$ and $M$.
\end{Lm}
\begin{proof}
    For a linear subset $V \subseteq \ell^2(\mu)$ let
    \[ Q_V = \bigcup_{v \in V} \{\Re(v),\ \Im(v)\}, \]
    where $\Re(v) = \{\Re(v_i)\}_{i \in \N}$ and $\Im(v) = \{\Im(v_i)\}_{i \in \N}$.
    Clearly, $V$ is symmetric if and only if $Q_V \subseteq V$.
    \newline
    \indent If $L = \{0\}$, put $M = \ell^2(\mu)$. Otherwise, let $c_0 \in L$ be a non-zero real-valued sequence, and let $\mathcal{B}_L$ be a Hamel basis of $L$ extending $\{c_0\}$ and consisting of real-valued sequences, by applying Theorem \ref{HamelBasis} with $P = \{c_0\}$ and $Q = Q_L$.
    Applying Theorem \ref{HamelBasis} again with $P = \mathcal{B}_L$ and $Q = Q_{\ell^2(\mu)}$, we can extend $\mathcal{B}_L$ to $\mathcal{B}_{\ell^2(\mu)}$ -- a Hamel basis of ${\ell^2(\mu)}$, consisting of real-valued sequences.
    Putting $M = \span(\mathcal{B}_{\ell^2(\mu)} \backslash \mathcal{B}_L)$ concludes the proof.
\end{proof}

\subsection{Construction of the space $\widetilde{\mathcal{H}}$}\label{ConstructionOfHTilde}

\begin{Df}
    For $\mu \in \mathcal{M}$, $w \in \C\backslash\supp\mu$ and $v \in \ell^2(\mu)$ define
    \[ D_w(v) = \{v_i (t_i - w)^{-1}\}_{i \in \N}. \]
\end{Df}

\begin{proof}[Proof of Theorem \ref{MainTheorem}]
    Let $W$ be an isometric isomorphism of $\mathcal{H}$ onto $\ell^2(\mu)$. For $v \in \ell^2(\mu)$ and $V \subseteq \ell^2(\mu)$ define
    \[ F_v = W^{-1}(v),\ \mathcal{F}_V = \Ran(W^{-1}|_V). \]
    Let $\mathcal{A}$ be the operator of multiplication by the independent variable, i.e. $\mathcal{A}v_n = t_n v_n$ for $v \in \ell^2(\mu)$, then $\mathcal{A}$ is an unbounded self-adjoint operator with dense in $\ell^2(\mu)$ domain $L$.
    \newline
    \indent By Lemma \ref{SymmetricLinearComplement} there exists a symmetric linear subset $M \subseteq \ell^2(\mu)$, with every element of $\ell^2(\mu)$ uniquely expressible as a sum of elements of $L$ and $M$.
    Since $\mathcal{A}$ is closed and unbounded, $M \ne \{0\}$. Therefore, there exist $z_0 \in \R\backslash\supp\mu$ and $u_0 \in M$ with $F_{u_0}(z_0) \ne 0$.
    Finally, for a sequence $u \in M$ let
    \[ G_u(z) = (z - z_0) F_{D_{z_0}(u)}(z). \]
    Put
    \[ \widetilde{\mathcal{H}} = \{H_v = G_u + F_c \mid v = u + c \in \ell^2(\mu), u \in M, c \in L\}. \]
    It is clear that the set $\widetilde{\mathcal{H}}$ with the norm $H_v \mapsto ||v||_{\ell^2(\mu)}$ is a Hilbert space of entire functions that is isometrically isomorphic to $\ell^2(\mu)$. Since $L$ is dense in $\ell^2(\mu)$, we have $\mathcal{F}_L \subseteq \mathcal{H} \cap \widetilde{\mathcal{H}}$ is a dense subset of $\mathcal{H}$ and $\widetilde{\mathcal{H}}$. By Remark \ref{Remark} we do not need to check the equality of norms in axioms (H1) and (H3) for the space $\widetilde{\mathcal{H}}$.
    \newline
    \textbf{$\widetilde{\mathcal{H}}$ satisfies (H3):}
    \newline
    \indent For every $v \in \ell^2(\mu)$ we have $H_v^* = G_u^* + F_c^* = G_{\bar{u}} + F_{\bar{c}} \in \widetilde{\mathcal{H}}$, since $L$, $M$ are symmetric and $\mathcal{H}$ satisfies (H3).
    \newline
    \textbf{$\widetilde{\mathcal{H}}$ satisfies (H1):}
    \newline
    \indent If $v \in \ell^2(\mu)$ and $H_v(w) = 0$ for some $w \in \C\backslash\R$, then
    \begin{equation*}
        \begin{aligned}
            &H_v(z)(z - w)^{-1} = (G_u(z) + F_c(z))(z - w)^{-1} = \\
            &((z - z_0) F_{D_{z_0}(u)}(z) + F_c(z))(z - w)^{-1} = \\
            &F_{D_{z_0}(u)}(z) + ((w - z_0) F_{D_{z_0}(u)}(z) + F_c(z))(z - w)^{-1} = \\
            &F_{D_{z_0}(u)}(z) + F_{(w - z_0) D_{z_0}(u) + c}(z)(z - w)^{-1}.
        \end{aligned}
    \end{equation*}
    Furthermore, $F_{(w - z_0) D_{z_0}(u) + c}(w) = (w - z_0) F_{D_{z_0}(u)}(w) + F_c(w) = G_u(w) + F_c(w) = H_v(w) = 0$
    and $(w - z_0) D_{z_0}(u) + c \in L$ as $D_{z_0}(u), c \in L.$
    \newline
    Applying the property (H1) to $F_{(w - z_0) D_{z_0}(u) + c} \in \mathcal{F}_L \subseteq \mathcal{H}$ we have
    \[ F_{(w - z_0) D_{z_0}(u) + c}(\cdot)(\cdot - w)^{-1} = F_{D_w((w - z_0) D_{z_0}(u) + c)} \in \mathcal{F}_L \subseteq \widetilde{\mathcal{H}},\]
    therefore $H_v(\cdot)(\cdot - w)^{-1} \in \widetilde{\mathcal{H}}$.
    \newline
    \textbf{$\widetilde{\mathcal{H}}$ violates (H2):}
    \newline
    \indent Since $L$ is dense $\ell^2(\mu)$ there exists a sequence $\{c_n\} \subseteq L$ converging to $u_0$ in $\ell^2(\mu)$.
    The space $\mathcal{H}$ satisfies (H2), therefore $F_{c_n}(z_0) \to F_{u_0}(z_0) \ne 0$, but $G_{u_0} \in \widetilde{\mathcal{H}}$ vanishes at $z_0$, so $F_{c_n}(z_0) \nrightarrow G_{u_0}(z_0)$.
\end{proof}

\subsection{Construction of a de Branges space that is isometrically isomorphic to $\ell^2(\mu)$}\label{ConstructionOfXv}

Before establishing Theorem \ref{SecondTheorem}, we prove a technical Lemma using the following statement.

\begin{Th}[Mittag-Leffler]\label{MittagLeffler}
    Let $\{a_n\}$ be a discrete subset of the complex plane. For each natural $n$ let
    \[ p_n(z) = \sum_{k = 1}^{K_n} \frac{c_{n, k}}{(z - a_n)^k}. \]
    Then there exists a meromorphic function $f$ whose poles are precisely $\{a_n\}$ and for every $n \in \N$ $f - p_n$ has a removable singularity at $a_n$.
    Moreover, if $c_{n, k}$, $a_n$ $\in \R$ for all $n, k$, then $f$ may be chosen such that $f^* = f$.
\end{Th}

\begin{Lm}\label{ExistenceOfS}
    For every $\mu \in \mathcal{M}$ there exists an entire function $S$, vanishing precisely on $\supp\mu$, all zeroes of which are simple, satisfying $S^* = S$ and
    \[ \sum_{n \in \N} \frac{1}{|S'(t_n)|^2 \mu_n} < +\infty. \]
\end{Lm}
\begin{proof}
    Pick $\{s_n\}$ -- a sequence of non-zero real numbers with $\sum_n \mu_n^{-1} |s_n|^{-2} < +\infty.$
    Let $\Pi$ be the Weierstrass canonical product for the sequence $\{t_n\},$ then $\Pi = \Pi^*$ is an entire function whose zero set coincides with $\supp\mu$.
    For each $n \in \N$ define $p_n(z) = c_n (z - t_n)^{-1}$, where $c_n = \Pi'(t_n)^{-1} \log|s_n / \Pi'(t_n)| \in \R$.
    \newline
    \indent By Theorem \ref{MittagLeffler} there exists a meromorphic function $f$ such that $f^* = f$ and $f$ having the residue $c_n$ at the point $t_n$ for every $n \in \N$.
    \newline
    \indent Let $T = \Pi f$, then $T$ is an entire function satisfying $T^* = T$ and $T(t_n) = \log|s_n / \Pi'(t_n)|$.
    Finally, put $S = \Pi e^T$, then $|S'(t_n)| = |s_n|$ and $S^* = \Pi^* e^{T^*} = S$.
\end{proof}

\begin{proof}[Proof of Theorem \ref{SecondTheorem}]
    Fix an entire function $S$ satisfying the conclusions of Lemma \ref{ExistenceOfS} for the measure $\mu$.
    Put
    \[ \mathcal{H} = \braces*{X_v(z) = \sum_{n \in \N} v_n \frac{S(z)}{z - t_n} \frac{1}{S'(t_n)} \ \bigg|\ v \in \ell^2(\mu)}. \]
    We shall verify that the space $\mathcal{H}$ satisfies the de Branges axioms (H1) -- (H3), then by Theorem \ref{EquivalenceOfDefinitions} $\mathcal{H}$ is a de Branges space.
    \newline
    \textbf{$\mathcal{H}$ is a Hilbert space of entire functions and is isometrically isomorphic to $\ell^2(\mu)$:}
    \newline
    \indent Let $K$ be a compact subset of the complex plane.
    Since $\supp\mu$ is a discrete set,
    $\max\limits_{n \in \N} \max\limits_{z \in K} |S(z)(z - t_n)^{-1}| \le C_K < +\infty$, therefore
    \[ \sup_{z \in K} |X_v(z)| \le C_K \sum_{n \in \N} |v_n| \sqrt{\mu_n} \frac{1}{|S'(t_n)| \sqrt{\mu_n}} \le C_K ||v||_{\ell^2(\mu)} \sqrt{\sum_{n \in \N} \frac{1}{|S'(t_n)|^2 \mu_n}} < +\infty. \]
    By the Weierstrass M-test, the series converges uniformly on $K$, so $X_v$ is continuous.
    Moreover, for a triangle $\gamma \subseteq \C$ we have
    \[
        \oint\limits_\gamma X_v(z) \dd z = \oint\limits_\gamma \sum_{n \in \N} v_n \frac{S(z)}{z - t_n} \frac{1}{S'(t_n)} \dd z =
        \sum_{n \in \N} \frac{v_n}{S'(t_n)} \oint\limits_\gamma \frac{S(z)}{z - t_n} \dd z = 0,
    \]
    so by Morera's Theorem $X_v$ is entire.
    \newline
    Notice that $X_v(t_n) = v_n$, so $X_v(t_n) = 0 \ \forall n$ iff $X_v \equiv 0$, therefore $X_v \mapsto ||v||_{\ell^2(\mu)}$ is a norm on $\mathcal{H}$, which is isometrically isomorphic to $\ell^2(\mu)$ by definition.
    \newline
    \textbf{$\mathcal{H}$ satisfies (H3):}
    \newline
    \indent Direct computation shows that $X_v^* = X_{\bar{v}} \in \mathcal{H}$, since $\bar{v} \in \ell^2(\mu)$.
    \newline
    \textbf{$\mathcal{H}$ satisfies (H1):}
    \newline
    \indent Suppose that $X_v(w) = 0$ for some $w \in \C\backslash\R$. Then $S(w) \ne 0$ and
        \begin{multline*}
            X_{D_w(v)}(z) = \sum_{n \in \N} v_n \frac{S(z)}{(z - t_n)(t_n - w)} \frac{1}{S'(t_n)} = \\
            \frac{1}{z - w}\parens*{\sum_{n \in \N} v_n \frac{S(z)}{z - t_n} \frac{1}{S'(t_n)} - S(z) \sum_{n \in \N} v_n \frac{1}{w - t_n} \frac{1}{S'(t_n)}} =
            \frac{X_v(z)}{z - w} - \cancelto{0}{\frac{S(z)}{(z - w) S(w)} X_v(w)}.
        \end{multline*}
    \newline
    \textbf{$\mathcal{H}$ satisfies (H2):}
    \newline
    \indent If $z = t_m$ for some $m \in \N$, then $|X_v(t_m)| \le \sqrt{\mu_m}^{-1} ||v||_{\ell^2(\mu)}$. Otherwise,
        \[ |X_v(z)| = \bigg|\sum_{n \in \N} v_n \sqrt{\mu_n} \frac{S(z)}{z - t_n} \frac{1}{S'(t_n) \sqrt{\mu_n}}\bigg| \le
            ||v||_{\ell^2(\mu)} |S(z)| \max_n |z - t_n|^{-1} \sqrt{\sum_{n \in \N} \frac{1}{|S'(t_n)|^2 \mu_n}}. \]
\end{proof}

\section{Acknowledgements}
The author is deeply grateful to Vladimir Kapustin for stating the initial problem, and to Roman Romanov for providing helpful suggestions and pointing out the connection between the author's undergraduate thesis and the de Branges conjecture.

\bibliographystyle{elsarticle-num-names}
\bibliography{bibliography}

\end{document}